\documentclass[12pt]{article}
\usepackage{amsmath,amsthm,amsfonts,amssymb}
\usepackage{color,epsfig}
\usepackage{fullpage}
\usepackage{enumerate}
\usepackage{paralist}
\usepackage{hyperref}
\usepackage{relsize}
\usepackage{exscale} 
\usepackage[normalem]{ulem}
\usepackage{soul}

\usepackage[english]{babel}

\usepackage{subfig}

\usepackage{pgf,tikz,pgfplots}

\pgfplotsset{compat=1.14}
\usepackage{mathrsfs}
\usetikzlibrary{calc}
\usetikzlibrary{arrows}
\usetikzlibrary{patterns} 
\usepackage{ifpdf}

\newtheorem{theorem}{Theorem}

\newtheorem{cor}[theorem]{Corollary}
\newtheorem{lemma}[theorem]{Lemma}
\newtheorem{remark}[theorem]{Remark}
 \newtheorem{observation}[theorem]{Observation}
 
\newcommand{\dist}{\mathrm{d_h}}
\newcommand{\distp}{\mathrm{d'_h}}
\newcommand{\distpp}{\mathrm{d''_h}}

\newcommand{\calA}{\mathcal{A}}
\newcommand{\calB}{\mathcal{B}}

\newcommand{\calL}{\mathcal{L}}

\newcommand{\calP}{\mathcal{P}}

\newcommand{\EE}{\mathbb{E}}
\newcommand{\HH}{\mathbb{H}}
\newcommand{\NN}{\mathbb{N}}

\newcommand{\RR}{\mathbb{R}}

\renewcommand{\Pr}{\mathbf{P}}
\newcommand{\Ex}{\mathbf{E}}

\newcommand{\paren}[1]{\left( \left. #1 \right. \right)} 
\newcommand{\cro}[1]{\left[ \left. #1 \right. \right]} 
\newcommand{\set}[1]{\left\{ \left. #1 \right. \right\}}
\newcommand{\e}[1]{\Ex\cro{#1}}
\newcommand{\p}[1]{\Pr\paren{#1}}

\newcommand{\poimod}{\mathrm{Poi}}

\newcommand{\red}[1]{\textcolor{red}{#1}}
\newcommand{\blue}[1]{\textcolor{blue}{#1}}
\newcommand{\green}[1]{\textcolor{green}{#1}}
\newcommand{\Tmax}{\theta^R}
\newcommand{\tmax}[1]{\Tmax\left(#1\right)}

\newcommand{\EC}{\mathbf{C}}

\begin{document}

\title{On the largest component of subcritical random hyperbolic graphs}

\author{Roland Diel \thanks{Universit\'e C\^ote d'Azur, CNRS, LJAD, France, Email: \texttt{roland.diel@univ-cotedazur.fr}.} \\
\and 
Dieter Mitsche\thanks{Institut Camille Jordan, Univ. Jean Monnet, Univ. St Etienne, Univ. Lyon, France, Email: \texttt{dmitsche@univ-st-etienne.fr}.  Dieter Mitsche has been supported by IDEXLYON of Universit\'{e} de Lyon (Programme Investissements d'Avenir ANR16-IDEX-0005).}}

\maketitle 

\begin{abstract}

We consider the random hyperbolic graph model introduced by~\cite{KPKVB10} and then formalized by~\cite{GPP12}. We show that, in the subcritical case $ \alpha > 1$, the size of the largest component is $n^{1/(2\alpha)+o(1)}$, thus strengthening a result of~\cite{BFM15} which gave only an upper bound of $n^{1/\alpha + o(1)}$.
\end{abstract}

\section{Introduction and statement of result}

In the last decade, the model of random hyperbolic graphs introduced by Krioukov et al. in ~\cite{KPKVB10} was studied quite a bit due to its key properties also observed in large real-world networks. In~\cite{BPK10} the authors showed empirically that the network of autonomous systems of the Internet can be very well embedded in the model of random hyperbolic graphs for a suitable choice of parameters. Moreover, Krioukov et al.~\cite{KPKVB10} gave empiric results that the model  exhibits the algorithmic small-world phenomenon established by the groundbreaking letter forwarding experiment of Milgram from the '60s~\cite{TM67}. From a theoretical point of view, the model of random hyperbolic graphs has an elegant specification and is thus amenable to rigorous analysis by mathematicians. Informally, the vertices are identified with points in the hyperbolic plane, and two vertices are connected by an edge if they are close in hyperbolic distance. 

A common way of visualizing the hyperbolic plane is via its native representation described in \cite{BKLMM17} where the choice for ground space is $\mathbb{R}^2$. Here, a point of $\mathbb{R}^2$ with polar coordinates $(r, \theta)$ has hyperbolic distance to the origin $O$ equal to its Euclidean distance $r$ and more generally, the hyperbolic distance $\dist(u,u')$
  between two points $u=(r_u, \theta_u)$ and $u'=(r_{u'}, \theta_{u'})$ is obtained by solving 
\begin{equation}\label{eqn:coshLaw}
\cosh \dist(u,u') := \cosh r_u\cosh r_{u'}-
  \sinh r_u\sinh r_{u'}\cos( \theta_u{-}\theta_{u'}).
\end{equation}
In the native representation, an instance of the graph can be drawn by mapping a vertex $v$ to the point in $\mathbb{R}^2$ with polar coordinate $(r_v, \theta_v)$ and drawing edges as straight lines (see Figure~\ref{Antoine}). 
\begin{figure}
\centering
    \includegraphics[width=0.55\textwidth]{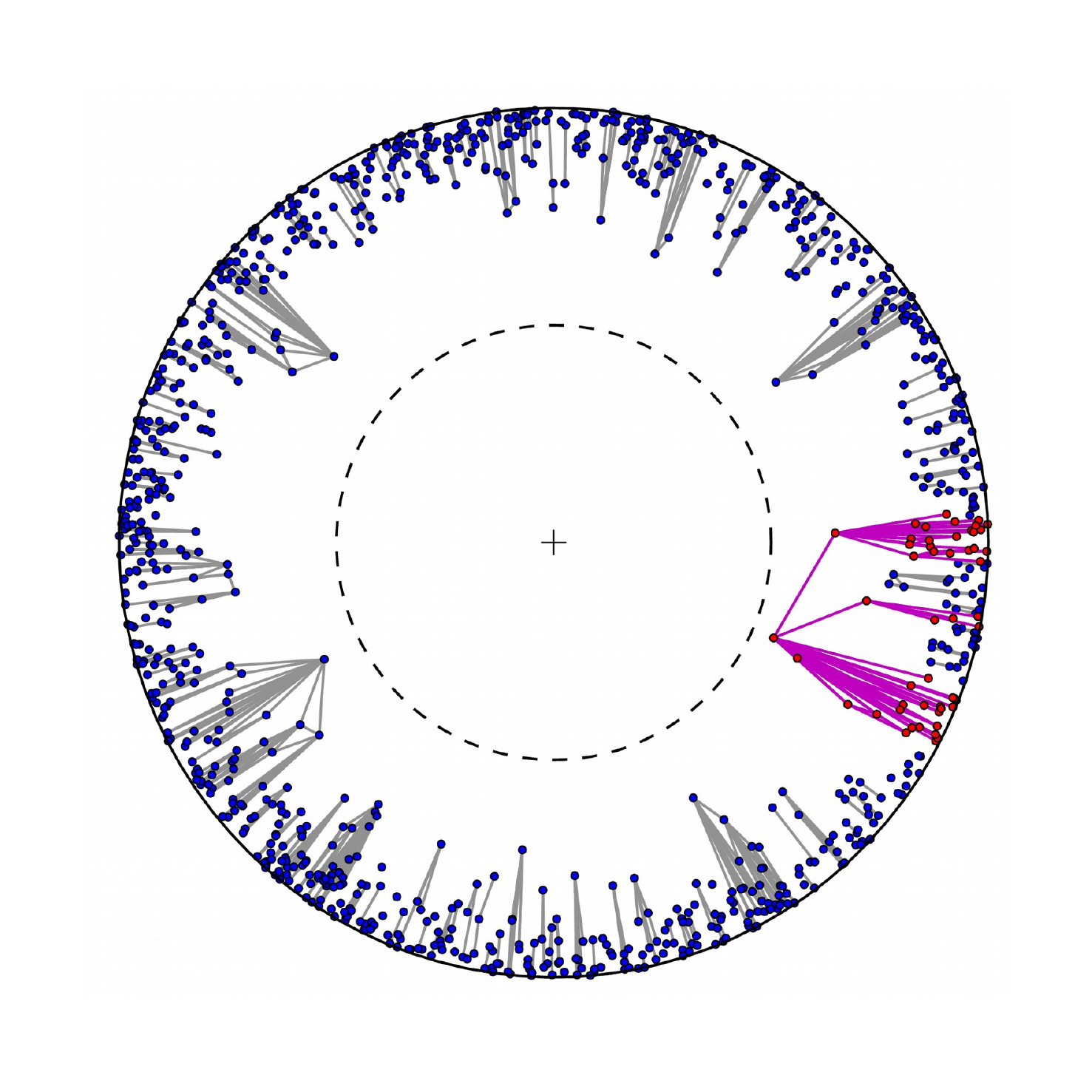}
\caption{A realization of the subcritical hyperbolic graph $\poimod_{\alpha,\nu}(n)$ with parameters $\alpha=1.1$, $\nu=1$, $n=1000$. The outer circle of the figure corresponds to $B(O,R)$, the inner dashed circle is $B(O,R/2)$.
 The size of the largest connected component, in purple, is  $|L_1|=51$.}\label{Antoine}
\end{figure}

The random hyperbolic model is defined as follows:
%\hspace{-5cm}
%\dm{choose randomly} $n$ points in the disk of radius $R=R(n)$ centered at the origin $O$ of the hyperbolic plane (denoted by $B_O(R)$ in the sequel) \dm{and identify them with vertices of a graph.} From a probabilistic point of view, it is arguably more natural to consider the Poissonized version of this model, \dm{and since for our purposes the models are equivalent, we stick to the Poissonized model.} Formally, this model 
%\mkc{Does [GPP12] actually talk about Poissonization? Needs checking.}\dmc{No. GPP12 does not Poissonize it. So I changed it a bit.} 
for each $n \in \NN$, we consider a Poisson point process on the disk $B_h(O,R)$ of the hyperbolic plane. The radius is equal to $R:=2 \log (n/\nu)$  for some positive constant $\nu \in \RR^+$ ($\log$ denotes here and throughout the paper the natural logarithm). The intensity function at polar coordinates $(r,\theta)$ for 
  $0\leq r< R$ and $0 \leq \theta < 2\pi$ is 
\[
g(r,\theta) := \nu e^{\frac{R}{2}}f(r,\theta)
\]
where $f(r,\theta)$ is the density function corresponding to the uniform probability on the disk $B_h(O,R)$ of the hyperbolic space of curvature $-\alpha^2$, that is $\theta$ is chosen uniformly at random in the interval $[0,2\pi)$ and independently of $r$ which is chosen according to the density function
\begin{align*}
f(r) & := \begin{cases}\displaystyle
   \frac{\alpha\sinh(\alpha r)}{\cosh(\alpha R)-1}, &\text{if $0\leq r< R$}, \\
   0, & \text{otherwise}.
  \end{cases}
\end{align*}
% The choice of $V$ is due to the fact that we will identify points of the Poisson process with vertices of the graph. 
 Make then the following graph $G=(V,E)$. The set of vertices $V$ is the points set of the Poisson process and for $u, u'\in V$, $u \neq u'$, there is an edge with endpoints 
  $u$ and $u'$ provided the distance (in the hyperbolic plane) between
  $u$ and $u'$ is at most $R$, i.e.,  
  the hyperbolic distance $\dist(u,u')$
  between $u$ and $u'$  is such that  $\dist(u,u')\leq R$, where $\dist(u,u')$ is obtained by solving Equation \eqref{eqn:coshLaw}
% \begin{equation}\label{eqn:coshLaw}
% \cosh \dist := \cosh r_u\cosh r_{u'}-
%   \sinh r_u\sinh r_{u'}\cos( \theta_u{-}\theta_{u'}).
% \end{equation}

For a given $n \in \NN$, we denote this model by 
  $\poimod_{\alpha,\nu}(n)$.
Note in particular that 
\[
\int g(r,\theta) d\theta dr 
  = \nu e^{\frac{R}{2}}=n,
\]
and thus  $\EE|{V}|=n$. In the original model of Krioukov et al.~\cite{KPKVB10}, $n$ points, corresponding to vertices, are chosen uniformly and independently in the disk $B_h(O,R)$ of the hyperbolic space of curvature $-\alpha^2$, but since from a probabilistic point of view it is arguably more natural to consider the Poissonized version of this model, we consider the latter one; see also~\cite{GPP12} for the construction of the uniform model.

  The restriction $\alpha>\frac12$ and the role of $R$ guarantee that the resulting graph has bounded average degree (depending
  on $\alpha$ and $\nu$ only). If $\alpha<\frac12$, then the degree sequence is so 
  heavy tailed that this is impossible (the graph is with high probability connected in this case, as shown in~\cite{BFM13b}). Moreover, if $\alpha>1$, then
  as the number of vertices grows,
  the largest component of a random hyperbolic graph has sublinear 
  order (see ~\cite[Theorem~1.4]{BFM15}).

\smallskip
\noindent\textbf{Notations:} 
We say that an event holds \emph{asymptotically almost surely (a.a.s.)}, if it holds with probability tending to $1$ as $n \to \infty.$ 
Given positive sequences $(a_n)_{n \geq 1}$ and
$(b_n)_{n \geq 1}$ taking
values in $\mathbb{R}$,
we write  $a_n = o(b_n)$  to mean that
 $a_n/b_n \to 0$ as $n \to \infty$.
 Also we write $a_n = \Theta(b_n)$ if $|a_n|/|b_n|$
is bounded away from 0 and $\infty$, and $a_n=\Omega(b_n)$ if $|a_n|/|b_n|$ is bounded away from $0$.

\smallskip
\noindent\textbf{Result:}
In this paper we study the size of the largest component of the graph in the case $\alpha>1$. In~\cite[Theorem~1.4]{BFM15} it was shown that its size is a.a.s. at most $n^{1/\alpha +o(1)}$. The main result of this paper is the following improvement, finding the exact exponent:
\begin{theorem}\label{thm:main}
Let $\alpha>1$ and $\nu\in\RR^+$. Let $G=(V,E)$ be chosen according to $\poimod_{\alpha,\nu}(n)$, and let $L_1 \subseteq G$ be the largest connected component of $G$. There is a constant $C>0$, such that, a.a.s., the following holds:
$$
n^{\frac{1}{2\alpha}}(\log n)^{-C}\leq |L_1| \leq n^{\frac{1}{2\alpha}}(\log n)^C\ .
$$

% $$
% |L_1| = n^{\frac{1}{2\alpha}+o(1)}\ .
% $$
% $$
% |C_1|=(1+o(1))\text{deg}_\text{max}\approx cn^{\frac{1}{2\alpha}}(1+o(1))
% $$
% with $c=\frac{2\alpha v^{1-\frac{1}{2\alpha}}}{\pi(\alpha-1/2)}$.
\end{theorem}
\begin{remark}
A careful inspection of the proofs shows that all results hold with probability at least $1-o(n^{-1/2})$, and hence a Depoissonization argument (see~\cite{MP03} for details) shows that Theorem~\ref{thm:main} also holds for the original uniform model.
\end{remark}
\smallskip
\noindent\textbf{Related work:} The size of the largest component in random hyperbolic graphs was first studied in~\cite{BFM15}: it was shown that for $\alpha > 1$ it is at most $n^{1/\alpha +o(1)}$, whereas for $\alpha < 1$ the largest component is linear. In the same paper the authors also showed that for $\alpha=1$ and $\nu$ sufficiently small there is a.a.s. no linear size component, whereas for $\alpha=1$ and $\nu$ sufficiently large a.a.s. there is a linear size component.  In~\cite{FMLaw} the picture was made more precise: for $\alpha=1$ there is a critical intensity such that a.a.s. a linear size component exists iff $\nu$ is above a certain threshold. Also, for $\alpha < 1$, for fixed $\alpha$ the size of the largest component is increasing in $\nu$, and for fixed $\nu$, it is decreasing in $\alpha$. Furthermore, in~\cite{BFM13b} it was shown that for $\alpha < 1/2$ the graph is connected a.a.s., whereas for $\alpha=1/2$ the probability of being connected tends to $1$ if $\nu \ge \pi$, and the probability of being connected is otherwise a monotone increasing function in $\nu$ that tends to $0$ as $\nu$ tends to $0$. For the case $1/2 < \alpha < 1$, it was shown in~\cite{KM19+} that a.a.s. the second component is of size $\Theta((\log n)^{1/(1-\alpha)})$, whereas for $\alpha=1/2$ and $\nu$ sufficiently small it is $\Theta(\log n)$ with constant probability, and for $\alpha=1$ it is a.a.s. $\Omega(n^b)$ for some $b > 0$. Starting with the seminal work of~\cite{KPKVB10}, further aspects of random hyperbolic graphs have been discussed since then: the power law degree distribution, mean degree and clustering coefficient were analyzed in~\cite{GPP12}; the diameter was computed in~\cite{fk15, km15, MStaps}, the spectral gap was analyzed in~\cite{KM18}, typical distances were calculated in~\cite{ABF}, and bootstrap percolation in such graphs was considered in~\cite{CF16}. First passage percolation of random hyperbolic graphs (or more generally, geometric inhomogeneous random graphs) was analyzed in~\cite{KL19+}.

%The paper is organized as follows: 
\smallskip
\noindent\textbf{Organization of the paper:} 
In Section~\ref{Preliminaries} we recall some well known properties of the random hyperbolic graph. Section~\ref{SepZon} then describes the construction of the main tool of our proof: the separation zones.  The existence of these zones shows that there is no long path of vertices with all vertices having roughly the same radial coordinates. Finally, in Section \ref{CovCom} we use the separation zones to control the size of the connected components of the graph which leads to the result of Theorem \ref{thm:main}.

\section{Preliminaries}\label{Preliminaries}
From now on, we suppose $\alpha>1$. In this section we collect some properties concerning random hyperbolic graphs. For notational convenience, for any point $v=(r_v,\theta_v)$ of the ball $B(O,R)$ we define $t_v=R-R_v$, the radial distance to the boundary circle of radius $R$ (instead of the distance to the origin $O$), and we identify a vertex $v$ of the graph $G$ with the coordinate pair $v=(t_v,\theta_v)$. Moreover, we suppose throughout the paper that $R$ is an integer. 
  
\medskip
By the hyperbolic law of cosines~\eqref{eqn:coshLaw}, 
  the hyperbolic triangle formed by the geodesics 
  between points $p'$, $p''$, and  $p$, with opposing side segments of length 
  $\distp$, $\distpp$, and $\dist$ respectively,
  is such that the angle formed at $p$ is:
\begin{equation*}\label{eqn:angle}
\theta_{\dist}(\distp,\distpp) = 
\arccos\Big(\frac{\cosh \distp\cosh \distpp-\cosh \dist}{\sinh \distp\sinh \distpp}\Big).  
\end{equation*}
Clearly, $\theta_{\dist}(\distp,\distpp) = \theta_{\dist}(\distpp,\distp)$. 
We state a very handy approximation for $\theta_{R}(\cdot,\cdot)$.
\begin{lemma}[{\cite[Lemma 3.1]{GPP12}}]\label{lem:aproxAngle}
If $0\leq\min\{\distp,\distpp\}\leq R \leq \distp+\distpp$, then
\[
\theta_{R}(\distp,\distpp) = 2e^{\frac{1}{2}(R-\distp-\distpp)}\big(1+\Theta(e^{R-\distp-\distpp})\big).
\]
\end{lemma}

A direct consequence of this lemma is the following corollary: 
\begin{cor}\label{lem:approxAnglet}
For any $R>0$, there is a function 
$$
\Tmax : \begin{array}{ccl}
    [0,R/2]^2&\to&\RR^+\\
    (t_1,t_2)&\mapsto&\tmax{t_1,t_2}
        \end{array}
$$ such that 
\begin{itemize}
 \item $\tmax{t_1,t_2} = 2e^{-\frac{1}{2}(R-t_1-t_2)}\big(1+\Theta(e^{-\frac{1}{2}(R-t_1-t_2)})\big)$
 \item two vertices $u,v\in V$ such that $t_u+t_v\leq R$ are connected by an edge iff $|\theta_u-\theta_v|\leq \tmax{t_u,t_v}.$
\end{itemize}
\end{cor}

Throughout, we will need estimates for measures of 
  regions of the hyperbolic plane, and more specifically, for regions obtained by performing some set algebra
  involving a few balls.
For a point $p$ of the hyperbolic plane $\HH^2$, 
  the ball of radius $\rho$ centered at $p$ will be denoted by
  $B_{p}(\rho)$, i.e., 
  $B_{p}(\rho) := \{q\in\HH^2 : \dist(p,q)\leq\rho\}$.

Also, we denote by $\mu(S)$ the measure of a set 
  $S \subseteq \HH^2$, i.e.,
  $\displaystyle\mu(S) := \int_{S}f(r,\theta)drd\theta$.

Next, we collect a few standard results for such measures.
\begin{lemma}[{\cite[Lemma~3.2]{GPP12}}]\label{lem:muBall} 
If $0\leq \rho< R$, then
  $\mu(B_{O}(\rho)) = \nu e^{-\alpha(R-\rho)}(1+o(1))$.
%Moreover, if $p\in B_{O}(R)$, then for 
%  $C_{\alpha}:=2\alpha/(\pi(\alpha-\frac{1}{2}))$,
%\begin{equation}\label{eq:muBall2}
%\mu(B_{p}(R)\cap B_{O}(R)) = C_{\alpha}e^{-\frac{r_p}{2}}\big(1+O(e^{-(\alpha-\frac{1}{2})r_p}+e^{-r_p})\big).
%\end{equation}
\end{lemma}

We  also use classical Chernoff concentration bounds for Poisson random variables. See for instance (\cite{MR3185193} page 23).
\begin{lemma}[Chernoff bounds]\label{lem:Chernoff} 
 If $X\sim \calP(\lambda)$, then for any $x>0$,
\begin{align*}
\p{X\geq \lambda+ x}&\leq e^{-\frac{x^2}{2(\lambda+x)}}\quad\text{ and }\quad
\p{X\leq \lambda- x}\leq e^{-\frac{x^2}{2(\lambda+x)}}.
 \end{align*}
 In particular, for $x\geq\lambda$,
 \begin{align*}
\p{X\geq 2x}&\leq e^{-\frac{x}{4}}%\quad\text{ and }\quad\p{X\leq \lambda/2}\leq e^{-\frac{\lambda}{12}}.
 \end{align*}
\end{lemma}

Lemma~\ref{lem:Chernoff} together with Lemma~3.2 of~\cite{GPP12} yield the following lemma:
\begin{lemma}\label{lem:degree}
Let $V$ be the vertex set of a graph chosen according to $\poimod_{\alpha,\nu}(n)$, and let $v$ be a vertex with $t_v > C \log R$ for $C$ sufficiently large. Then, a.a.s. ${|V\cap B_{v}(R)|}=\Theta(e^{\frac12t_v})$.
\end{lemma}

%%%%%%%%%%%%%%%%%%%%%%%%%%%%%%%%%%%%%%%%%%%%%%%%%%%%%%%%%%%%%%%%%%%%%%%%%

\section{Construction of the separation zones}\label{SepZon}

In this section we explain how to construct the separation zones. We first define the following sectors
$$
S(\theta_1,\theta_2)=\set{(t,\theta) \mid 0\leq t< R \text{ and } \theta_1\leq\theta<\theta_2}
$$
and the annuli 
$$
\calL(t^-,t^+)=\set{(t,\theta) \mid t^-\leq t< t^+ \text{ and } 0\leq\theta<2\pi}.
$$
The following observation is a simple consequence of Lemma~\ref{lem:muBall}:
\begin{observation}\label{cor:medidalb} 
For any $0\leq t^-<t^+<R/2$
\[
\e{|V\cap\calL(t^-,t^+)|} = \nu e^{\frac{R}{2}-\alpha t^-}(1-e^{-\alpha(t^+-t^-)}+o(1)).
\]
\end{observation}

We then construct for each coordinate pair $(t_0,\theta_0)\in(0,R/2)\times[0,2\pi)$, a zone  that separates points to the left from points to the right in $\set{(t,\theta),t\leq t_0}$.
Precisely, define  for $t_0<R/2$ and $\theta_0\in[0,2\pi)$, the following \emph{separation zone}:
$$
\calA(t_0,\theta_0)=\set{(t,\theta) \mid t\leq t_0\text{ and }|\theta-\theta_0|\leq \tmax{t,t}}.
$$

We thus have the following observation:
\begin{observation}\label{obs:easy}
Suppose $V\cap\calA(t_0,\theta_0)=\emptyset$. Let $v,w \in \set{(s,\theta),s\leq t_0}$ with $\theta_v < \theta_0 < \theta_w$. Then $|\theta_v - \theta_w| > \theta^R(t_v, t_w)$, i.e. $v$ and $w$ are not connected by an edge.
\end{observation}
\begin{proof}
The function $\theta^R(t_v, t_w)$ is increasing in both of its arguments, hence we may assume that $t_v=t_w=t$. For this choice of $t$, $v$ and $w$ are connected by Corollary~\ref{lem:approxAnglet} iff $|\theta_v - \theta_w| \le \theta^R(t,t)$.
However, since $\calA(t_0,\theta_0)=\emptyset$ and $\theta_v < \theta_0 < \theta_w$, we have  $|\theta_v - \theta_w| > 2\theta^R(t, t)$, i.e. $v$ and $w$ are not connected by an edge.
\end{proof}

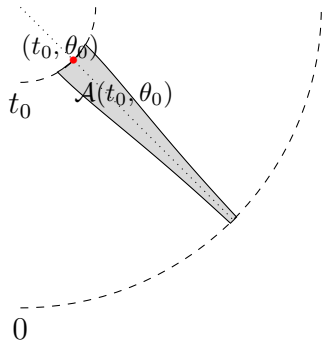
\begin{figure}[ht!]
\centering

\begin{tikzpicture}

\draw [domain=1:4] plot ({40*exp(-\x)-45}:\x);
\draw [domain=1:4] plot ({-40*exp(-\x)-45}:\x);

 \coordinate (A) at  ({-40*exp(-1)-45}:1)  ;
 \coordinate (B) at ({40*exp(-1)-45}:1)  ;
  \coordinate (C) at  ({40*exp(-4)-45}:4)  ;
 \coordinate (D) at ({-40*exp(-4)-45}:4)  ;
 
  \draw[thick] (A)-- (A) arc ({-40*exp(-1)-45}:{40*exp(-1)-45}:1) -- plot [domain=1:4] ({40*exp(-\x)-45}:\x) -- (C) arc ({40*exp(-4)-45}:{-40*exp(-4)-45}:4) -- plot [domain=4:1] ({-40*exp(-\x)-45}:\x) -- cycle ;

 \fill[color=gray!30] (A)-- (A) arc ({-40*exp(-1)-45}:{40*exp(-1)-45}:1) -- plot [domain=1:4] ({40*exp(-\x)-45}:\x) -- (C) arc ({40*exp(-4)-45}:{-40*exp(-4)-45}:4) -- plot [domain=4:1] ({-40*exp(-\x)-45}:\x) -- cycle ; 

\draw[dashed] (0,-1) arc (-90:0:1) ;
\draw[dashed] (0,-4) arc (-90:0:4) ;

\draw[dotted] (0,0) --  (-45:4)  ;
\draw (-45:1) node[red] {\tiny $\bullet$} ;
\draw (-45:0.75) node {\footnotesize $(t_0,\theta_0)$} ;
\draw (-90:1) node[below ] {$t_0$} ;
\draw (-90:4) node[below] {$0$} ;
\draw  (-40:1.8)  node {\footnotesize $\calA(t_0,\theta_0)$} ;

\end{tikzpicture}
\caption{a separation zone}
\end{figure}

To use the previous observation, we need separation zones which do not contain any vertices. We prove below that this happens with large probability. 
\begin{lemma} \label{zone1} For any $K>0$, there is a constant $c>0$ which depends only on $\alpha$ and $\nu$ such that for any $t< R/2$,
 $$
 \p{\exists j\in \set{0,\dots, c R} , V\cap\calA(t,2j\tmax{t,t})=\emptyset}\geq 1-e^{-KR}\ .
 $$
\end{lemma}
\begin{proof}
Consider the event 
$$
E=\set{\exists j\in \set{0,\dots,N } , V\cap\calA(t,2j\tmax{t,t})=\emptyset}
$$
for some $N$ that we will choose below. We recall that the set $\calA(t,2j\tmax{t,t})$ is included in the sector 
$$
S((2j-1)\tmax{t,t},(2j+1)\tmax{t,t}).$$
For different values of $j$, these sectors are disjoint, and thus the random variables  $|V\cap\calA(t,2j\tmax{t,t})|$ are independent and
\begin{align*}
 \p{\overline{E}}=  \paren{\p{|V\cap\calA(t,0)|>0}}^{N+1}= \paren{1-e^{-\e{|V\cap\calA(t,0)|}}}^{N+1}%\leq \paren{\e{|V\cap\calA(0)|}}^{2N+1}
\end{align*}
Then, as $t<R/2$, Corollaries \ref{lem:approxAnglet}  and \ref{cor:medidalb} give
\begin{align*}
 \e{|V\cap\calA(t,0)|}&\leq\sum_{1\leq i\leq \lceil t\rceil }2\tmax{i,i}\e{|V\cap \calL(i-1,i)|}\\
 &=\sum_{1\leq i\leq \lceil t\rceil} 4e^{-\frac{R}2+i}\nu e^{R/2}e^{-\alpha (i-1)}(1-e^{-\alpha})(1+o(1))\\
 &=4\nu(e^{\alpha}-1)\sum_{1\leq i\leq \lceil t\rceil}  e^{-(\alpha-1)i}(1+o(1))\\
 &\leq 4 \nu \frac{(e^{\alpha}-1)e^{-(\alpha-1)}}{1-e^{-(\alpha-1)}}+o(1). 
\end{align*}
By choosing $N= cR$ for some constant $c >0$ sufficiently large (depending on $K$) the lemma follows. 
\end{proof}

 We will now consider layers starting from the boundary of $B_O(R)$: set 
$$
\forall i\geq 0,\ t_i=\left(\frac{4\alpha}{\alpha-1}+3i\right)\log R.
$$

Let $t_{\max}=\frac{1}{2\alpha}R$ (note that $t_{\max} < R/2$), the distance to circle of radius $R$ roughly corresponding to the largest $t$ for which we can find an element of $V$ and set $i_{\max}=\min\set{i\geq0,\ t_i\geq t_{\max}}$. We thus have
$$i_{\max}\leq R \text{ and }t_{\max}\leq t_{i_{\max}}\leq t_{\max}+3\log R.$$

We also set $t_{-1}=0$ and we define, for $i,j\in\set{0,\dots,i_{\max}}$, the angle 
$$
\theta_{i,j}=\tmax{t_i,t_j}
$$ and the consecutive layers $$\calL_i= \calL(t_{i-1},t_{i}).$$
%It will become clear later why the first layer has to be larger than the other ones ($\sqrt{R}$ instead of $s(R)$). \dmc{peut-etre changer un peu si on choisit $s=\varepsilon \sqrt{R}$}

\begin{observation}\label{obs:Ltheta} 
For any $i,j\in\set{0,\dots,i_{\max}}$,
\[
\e{|V\cap \calL_i|} = \nu e^{\frac{R}{2}-\alpha t_{i-1}}(1+o(1))
\quad\text{and}\quad
\theta_{i,j}=2e^{-\frac{1}{2}(R-t_i-t_j)}\big(1+\Theta(e^{-\frac{1}{2}(R-t_i-t_j)})\big).
\]
\end{observation}

We now define the following separation zones: for every $i\in\set{0,\dots,i_{\max}}$, set $k_{\max}^i=\lceil 2\pi/(3cR\theta_{i,i})\rceil$ where $c$ is the constant given in Lemma \ref{zone1} for $K=1$. For every $0\leq k< k_{\max}^i $, we find the $(k+1)$-th separation zone to be the closest (to the right) empty region to the angle $3cRk\theta_{i,i}$. More formally, define for $0\leq k < k^i_{\max}$,
\begin{align*}
   j^{i,k}&=\min\set{j\in \NN \mid  V\cap\calA(t_i,(3cRk+2j)\theta_{i,i})=\emptyset}.
   \end{align*}
  We assign $\calA^{i,k}$ then to be the closest region  to $3cRk\theta_{i,i}$: 
  \begin{align*}
\calA^{i,k}&=\calA(t_i,(3cRk+2j^{i,k})\theta_{i,i}), 
\end{align*}
where $\min\emptyset=\infty$ and in this case $\calA^{i,k}=\emptyset$. 
The set $\calA^{i,k}$ represents the $(k+1)$-th separation zone of layer $i$. For notational convenience, we also set $\calA^{i,k^i_{\max}}=\calA^{i,0}$. We could have $\calA^{i,k}=\calA^{i,k+1}$, and the two sets might not even be well defined. We will thus use Lemma~\ref{zone1} to show that asymptotically almost surely none of the two things happens.

In order to state the next lemma properly, we define the following (pseudo)distance between separation zones:
$$
\forall A,B\subset B_O(R),\ d(A,B)=\inf\set{|\theta-\theta'|\ |\ (t,\theta)\in A,\ (t',\theta')\in B}.
$$
\begin{lemma}\label{lem:sep} Let $c$ be the constant given in Lemma \ref{zone1} for $K=1$ (depending only on $\alpha$ and $\nu$). Then the event $\mathcal{E}_R$ defined by
 $$
 \mathcal{E}_R=\set{\forall 0\leq i\leq i_{\max},\ \forall 0\leq k< k_{\max}^i,\   \calA^{i,k}\neq\emptyset \text{ and } cR\theta_{i,i} \leq d(\calA^{i,k},\calA^{i,k+1})\leq  5cR\theta_{i,i}}
 $$
 occurs a.a.s.

 \end{lemma}
\begin{proof}
Let $c$ be the constant given in Lemma~\ref{zone1} for $K=1$ and consider the event
$$
\mathcal{F}_R=\set{\forall 0\leq i\leq i_{\max},\ \forall 0\leq k< k_{\max}^i,\   \exists j\in \set{0,\dots,cR } , V\cap\calA(t_i,(3cRk+2j)\theta_{i,i})=\emptyset}.
$$
Clearly, $\mathcal{F}_R\subset \mathcal{E}_R$ and it is sufficient to bound $\p{\overline{ \mathcal{F}_R}}$.
Then, for $R$ large enough, using the definition of $k_{\max}^i$, 
  \begin{align*}
  \p{\overline{ \mathcal{F}_R}} \leq&\sum_{\substack{0\leq i\leq i_{\max}\\0\leq k< k_{\max}^i}}\p{\forall j\in \set{0,\dots,cR } , V\cap\calA(t_i,(3cRk+2j)\theta_{i,i})\neq\emptyset}\\
=&\sum_{0\leq i\leq i_{\max}}k_{\max}^i\p{\forall j\in \set{0,\dots, cR } , V\cap\calA(t_i,2j\theta_{i,i})\neq\emptyset}\\
   \leq& C_1\sum_{0\leq i\leq i_{\max}}e^{\frac{R}2-t_i}e^{-R}\leq C_2 e^{-\frac{1}2R}.
  % \leq&2\pi \sum_{0\leq i\leq i_{\max}}e^{\frac{R}2-is-\sqrt{R}-KR}(1+o(1))\leq C_1 e^{-(K-\frac{1}2)}
   \end{align*}
  Since the last quantity goes to zero as $R$ tends to infinity, the lemma is proven. 
\end{proof}

\begin{figure}[ht!]
\centering
\begin{tikzpicture}

\newcommand{\sepzone}{
\coordinate (A) at (-0.5,0)  ;
\coordinate (B) at (0.5,0)  ;
\coordinate (C) at ({exp(-3)/2},-3)  ;
\coordinate (D) at ({-exp(-3)/2},-3)  ;
  \fill[color=gray!30] (A)-- (B) -- plot [domain=0:-3] ({exp(\x)/2},\x) -- (D)  -- plot [domain=-3:0] ({-exp(\x)/2},\x) -- cycle ;
 \draw (A)-- (B) -- plot [domain=0:-3] ({exp(\x)/2},\x) -- (D)  -- plot [domain=-3:0] ({-exp(\x)/2},\x) -- cycle ;
 }

\begin{scope}[shift={(2,0)}]
\sepzone
\end{scope}

\begin{scope}[shift={(7,0)}]
\sepzone
\end{scope}

\begin{scope}[shift={(11,0)}]
\sepzone
\end{scope}

\begin{scope}[shift={(14,0)}]
\sepzone
\end{scope}

\draw[dashed] (-0.25,0) node[below]{$t_i$} -- (15.5,0);

\draw[dotted] (0,0) -- (0,-3) ;
\draw[dotted] (5,0) -- (5,-3) ;
\draw[dotted] (10,0) -- (10,-3) ;

\draw[dotted] (0,0) -- (0,-3) ;

\draw[dotted, <->] (0,-3.2) -- (5,-3.2) node[below, midway] {$3cR\theta_{i,i}$} ;
\draw[dotted, <->] (5,-3.2) -- (10,-3.2)  node[below, midway] {$3cR\theta_{i,i}$};

\draw (2.1,-0.2) node[below] {\tiny $\calA^{i,0}$} ;
\draw (7.1,-0.2) node[below] {\tiny $\calA^{i,1}$} ;
\draw (11.2,-0.2) node[below] {\tiny $\calA^{i,2}$} ;
\draw (14.2,-0.2) node[below] {\tiny $\calA^{i,k_{\max}^i-1}$} ;
\draw (1,-1.25) node[below] {$\calB^{i,0}$} ;
\draw (5,-1.25) node[below] {$\calB^{i,1}$} ;
\draw (9.5,-1.25) node[below] {$\calB^{i,2}$} ;
\draw (12.5,-1.5) node[below] {$\dots$} ;
\draw (15,-1.25) node[below] {$\calB^{i,0}$} ;
\end{tikzpicture}
\caption{The separation zones}
\end{figure}
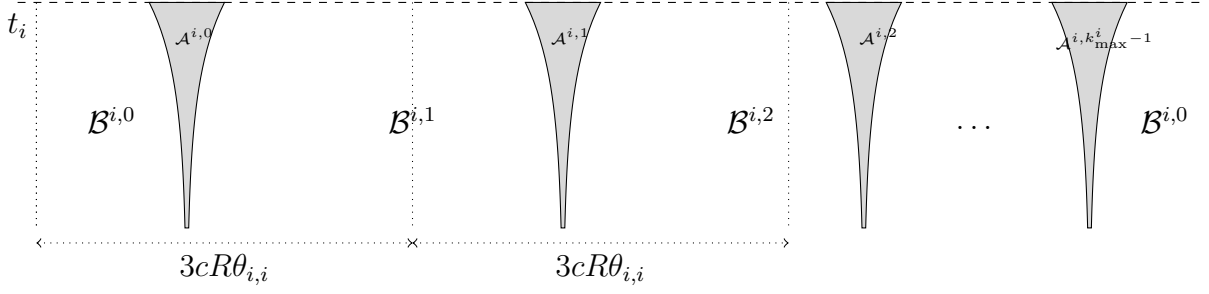

Hence, a.a.s. the distance between two consecutive separation zones $\calA^{i,k}$ and $\calA^{i,k+1}$ is always of the order $R\theta_{i,i}$. Define now, on  $\mathcal{E}_R$,  the area $\calB_{i,k}$ between two separation zones: for $1\leq k<k^i_{\max}$,

$$
\calB_{i,k}=\set{(t,\theta)\in B_O(R) \mid  t\leq t_i\text{ and }\forall (t,\theta_-)\in\calA^{i,k-1},\  \theta>\theta_-\ \text{ and }\  \forall (t,\theta_+)\in\calA^{i,k},\  \theta<\theta_+}
$$
and 
$$
\calB_{i,0}=\set{(t,\theta)\in B_O(R) \mid t\leq t_i\text{ and }\forall (t,\theta_-)\in\calA^{i,k^i_{\max}-1},\  \theta>\theta_-\ \text{ or }\  \forall (t,\theta_+)\in\calA^{i,0},\  \theta<\theta_+}
$$

We point out that every path of connected points from $u\in V\cap\calB_{i,k}$ to $v\in V\cap\calB_{i,\ell}$ with $k\neq \ell$ has to go through a vertex $w\in V$ such that $t_w>t_i$, i.e. points cannot be connected "below" as described in Figure \ref{figpoints}. 
More formally, rewriting Observation~\ref{obs:easy} we obtain the following observation.
\begin{observation}\label{easy:calB}
Suppose $\calA^{i,k}=\emptyset$. Let $u\in V\cap\calB_{i,k}$ to $v\in V\cap\calB_{i,\ell}$ with $k \neq \ell$.  Then $u$ and $v$ can only be connected by a path that has at least one intermediate vertex $w\in V$ such that $t_w>t_i$.
\end{observation}
\begin{proof}
Note that $t_u, t_v \le t_i$. Then by Observation~\ref{obs:easy}, $u$ and $v$ are not connected by an edge. Since this holds for any $k \ne \ell$ and any $u,v$, there can be no path between $u$ and $v$ containing vertices $w \in V$ such that $t_w \le t_i$ (see also Figure \ref{figpoints}).
\end{proof}

 \begin{figure}[ht!]
\centering
 \begin{tikzpicture}

 \newcommand{\sepzoneA}{
 \coordinate (A) at (-0.5,0)  ;
\coordinate (B) at (0.5,0)  ;
\coordinate (C) at ({exp(-3)/2},-3)  ;
\coordinate (D) at ({-exp(-3)/2},-3)  ;
  \fill[color=gray!30] (A)-- (B) -- plot [domain=0:-3] ({exp(\x)/2},\x) -- (D)  -- plot [domain=-3:0] ({-exp(\x)/2},\x) -- cycle ;
 \draw (A)-- (B) -- plot [domain=0:-3] ({exp(\x)/2},\x) -- (D)  -- plot [domain=-3:0] ({-exp(\x)/2},\x) -- cycle ;
 }
 
  \newcommand{\sepzoneB}{
   \coordinate (A) at ({-exp(-1)/2},-1)  ;
\coordinate (B) at ({exp(-1)/2},-1)  ;
\coordinate (C) at ({exp(-3)/2},-3)  ;
\coordinate (D) at ({-exp(-3)/2},-3)  ;
  \fill[color=gray!30] (A)-- (B) -- plot [domain=-1:-3] ({exp(\x)/2},\x) -- (D)  -- plot [domain=-3:-1] ({-exp(\x)/2},\x) -- cycle ;
 \draw (A)-- (B) -- plot [domain=-1:-3] ({exp(\x)/2},\x) -- (D)  -- plot [domain=-3:-1] ({-exp(\x)/2},\x) -- cycle ;
 }
 
 \newcommand{\sepzoneC}{
   \coordinate (A) at ({-exp(-2)/2},-2)  ;
\coordinate (B) at ({exp(-2)/2},-2)  ;
\coordinate (C) at ({exp(-3)/2},-3)  ;
\coordinate (D) at ({-exp(-3)/2},-3)  ;
  \fill[color=gray!30] (A)-- (B) -- plot [domain=-2:-3] ({exp(\x)/2},\x) -- (D)  -- plot [domain=-3:-2] ({-exp(\x)/2},\x) -- cycle ;
 \draw (A)-- (B) -- plot [domain=-2:-3] ({exp(\x)/2},\x) -- (D)  -- plot [domain=-3:-2] ({-exp(\x)/2},\x) -- cycle ;
 }
 
%Aires de separation S_{i,k}
\begin{scope}[shift={(0,0)}]
\sepzoneA
\end{scope}
 
\begin{scope}[shift={(4,0)}]
\sepzoneA
\end{scope}

\begin{scope}[shift={(11,0)}]
\sepzoneA
\end{scope}

%Aires de separation S_{i-1,k}
\begin{scope}[shift={(1,0)}]
\sepzoneB
\end{scope}

\begin{scope}[shift={(2.25,0)}]
\sepzoneB
\end{scope}

\begin{scope}[shift={(3,0)}]
\sepzoneB
\end{scope}

\begin{scope}[shift={(5,0)}]
\sepzoneB
\end{scope}

\begin{scope}[shift={(6.5,0)}]
\sepzoneB
\end{scope}

\begin{scope}[shift={(8.25,0)}]
\sepzoneB
\end{scope}

\begin{scope}[shift={(9,0)}]
\sepzoneB
\end{scope}

\begin{scope}[shift={(10,0)}]
\sepzoneB
\end{scope}

\newcommand{\sepzoneCC}{
\begin{scope}[shift={(0.25,0)}]
\sepzoneC
\end{scope}

\begin{scope}[shift={(0.5,0)}]
\sepzoneC
\end{scope}
}

\newcommand{\sepzoneCCC}{
\sepzoneCC

\begin{scope}[shift={(0.75,0)}]
\sepzoneC
\end{scope}
}

\newcommand{\sepzoneCcinq}{
\sepzoneCCC
\begin{scope}[shift={(1,0)}]
\sepzoneC
\end{scope}
\begin{scope}[shift={(1.25,0)}]
\sepzoneC
\end{scope}
}

\newcommand{\sepzoneCsix}{
\sepzoneCcinq
\begin{scope}[shift={(1.45,0)}]
\sepzoneC
\end{scope}
}

\sepzoneCCC

\begin{scope}[shift={(1,0)}]
\sepzoneCCC
\begin{scope}[shift={(0.95,0)}]
\sepzoneC
\end{scope}
\end{scope}

\begin{scope}[shift={(2.25,0)}]
\sepzoneCC
\end{scope}

\begin{scope}[shift={(3,0)}]
\sepzoneCCC
\end{scope}

\begin{scope}[shift={(4,0)}]
\sepzoneCCC
\end{scope}

\begin{scope}[shift={(5,0)}]
 \sepzoneCcinq
\end{scope}

\begin{scope}[shift={(6.5,0)}]
 \sepzoneCsix
\end{scope}

\begin{scope}[shift={(8.25,0)}]
 \sepzoneCC
\end{scope}

\begin{scope}[shift={(9,0)}]
 \sepzoneCCC
\end{scope}

\begin{scope}[shift={(10,0)}]
 \sepzoneCCC
\end{scope}

% \draw (0.35,-0.5) node[below] {\tiny $\calA^{i,k}$} ;
% \draw (4.35,-0.5) node[below] {\tiny $\calA^{i,k+1}$} ;
% \draw (11.35,-0.5) node[below] {\tiny $\calA_{i,k+2}$} ;

\draw[dashed] (-1,0) node[left]{$t_i$}-- (12.5,0) ;
\draw[dashed] (-1,-1) node[left]{$t_{i-1}$}-- (12.5,-1) ;
\draw[dashed] (-1,-2) node[left]{$t_{i-2}$}-- (12.5,-2) ;
\draw[dashed] (-1,-3) node[left]{$t_{i-3}$}-- (12.5,-3) ;

\draw[dashed, green] (0.8,-1.5) node{\green{$\times$}} -- (1.2,-0.25) node{\green{$\times$}} -- (1.65,-1.8) node{\green{$\times$}};
\draw[dashed, red] (4.85,-2.5) node{\red{$\times$}} -- (5.8,-1.25) node{\red{$\times$}} ;
\end{tikzpicture}
\caption{The green points are connected while the red ones are not.}\label{figpoints}
\end{figure}
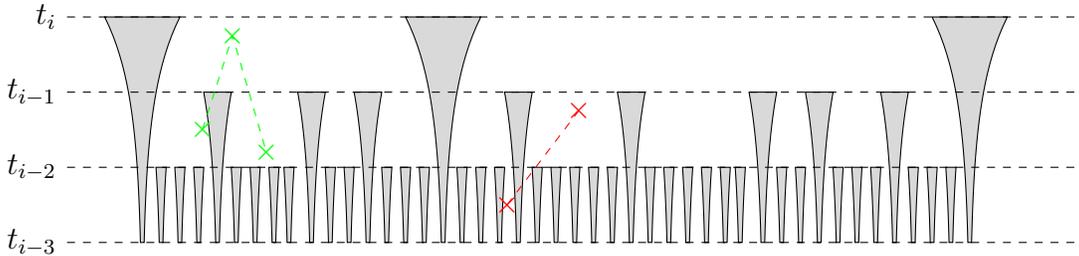

\section{Covering component} \label{CovCom}

On a high level, the advantage of separation zones is that it is impossible to stay in the same connected component going from right to left (or the other direction) remaining always at the same radius or going towards the boundary. We will thus construct, starting from a certain vertex, a \emph{covering} component, that is, a component which covers a.a.s. the whole connected component of the vertex if this vertex is the vertex closest to the center of its connected component.

We describe now in detail the iterative construction process of the covering component. Suppose that the event $\mathcal{E}_R$ holds. This happens a.a.s. according to Lemma \ref{lem:sep}. Consider a vertex $v\in V$. If $v$ is in the layer $\calL_0$, we define $C_v=\set{v}$ and if $v\in \calL_i$ for $0\leq j<i\leq i_{\max}$, we set 
$$\Theta_{i,j}(v)=V\cap \calL_j\cap S(\theta_v-2\theta_{i,j},\theta_v+2\theta_{i,j})$$
and (see also Figure~\ref{fig:Cv})
\begin{align*}
C_v&=\set{v}\cup\bigcup_{j=0}^{i-1}\bigcup_{u\in \Theta_{i,j}(v)}C_u.
\end{align*}
Denote now by $k$ the unique integer such that $v\in \calL_i\cap \calB_{i,k}$. The \emph{covering component} of $v$ is then defined as

$$
\EC_v=\bigcup_{u \in V\cap\calL_i\cap \calB_{i,k}}C_u.
$$

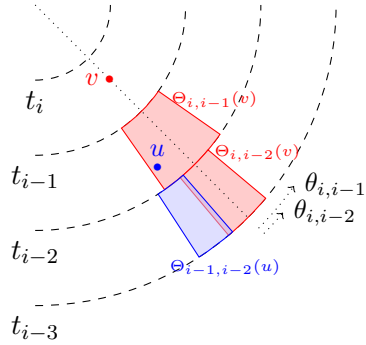
\begin{figure}[ht!]
\centering
 %\definecolor{ududff}{rgb}{0.30196078431372547,0.30196078431372547,1}
%\definecolor{ududff}{rgb}{0.30196078431372547,0.30196078431372547,1}
\begin{tikzpicture}
\coordinate (A1) at (-60:1)  ;
\coordinate (B1) at (-60:2)  ;
\coordinate (C1) at (-30:2)  ;
\coordinate (D1) at (-30:1)  ;
\coordinate (A2) at (-55:2)  ;
\coordinate (B2) at (-55:3)  ;
\coordinate (C2) at (-35:3)  ;
\coordinate (D2) at (-35:2)  ;
\coordinate (A3) at (-50:3)  ;
\coordinate (B3) at (-50:4)  ;
\coordinate (C3) at (-40:4)  ;
\coordinate (D3) at (-40:3)  ;

\draw[dashed] (0,-1) arc (-90:0:1) ;
\draw[dashed] (0,-2) arc (-90:0:2) ;
\draw[dashed] (0,-3) arc (-90:0:3) ;
\draw[dashed] (0,-4) arc (-90:0:4) ;

%\fill[color=gray!30] (A1) -- (B1) -- (B1) arc (-60:-30:2) -- (D1) --  (D1) arc (-30:-60:1)-- cycle ;
\fill[color=red!20] (A2) -- (B2) -- (B2) arc (-55:-35:3) -- (D2) --  (D2) arc (-35:-55:2)-- cycle ;
\draw[red] (A2) -- (B2) -- (B2) arc (-55:-35:3) -- (D2) --  (D2) arc (-35:-55:2)-- cycle ;
\fill[color=red!20] (A3) -- (B3) -- (B3) arc (-50:-40:4) -- (D3) --  (D3) arc (-40:-50:3)-- cycle ;
\draw[red] (A3) -- (B3) -- (B3) arc (-50:-40:4) -- (D3) --  (D3) arc (-40:-50:3)-- cycle ;
\draw  (-28:2.7)  node[red] {\tiny $\Theta_{i,i-1}(v)$} ;

%\draw (A2) -- (B2) -- (B2) arc (-55:-50:3) -- (B3) -- (B3) arc (-50:-40:4) -- (D3) -- (D3) arc (-40:-35:3) -- (D2) -- (D2) arc (-35:-55:2)  -- cycle ;

\draw[dotted] (0,0) --  (-45:4)  ;
\draw (-45:1.4) node[red] {\tiny $\bullet$} node[left, red] {\footnotesize $v$};

\draw (-90:1) node[below ] {\small $t_i$} ;
\draw (-90:2) node[below] {\small $t_{i-1}$} ;
\draw (-90:3) node[below] {\small $t_{i-2}$} ;
\draw (-90:4) node[below] {\small $t_{i-3}$} ;

\draw (-53:2.7) node[blue] {\tiny $\bullet$} node[above, blue] {\footnotesize $u$};
\fill[color=blue!25, opacity=0.5] (-57:3) -- (-57:4) -- (-57:4) arc (-57:-49:4) -- (-49:3) --  (-49:3) arc (-49:-57:3)-- cycle ;
\draw[blue] (-57:3) -- (-57:4) -- (-57:4) arc (-57:-49:4) -- (-49:3) --  (-49:3) arc (-49:-57:3)-- cycle ;
\draw  (-52:4.1)  node[below,blue] {\tiny $\Theta_{i-1,i-2}(u)$} ;
\draw  (-37:3.7)  node[above, red] {\tiny $\Theta_{i,i-2}(v)$} ;

\draw[dotted, ->] (-45:4.2) arc (-45:-35:4.2) ;
\draw (-35:4.2) node[ right] {\footnotesize $\theta_{i,i-1}$} ;
\draw[dotted, ->] (-45:4.3) arc (-45:-40:4.3) ;
\draw (-40:4.3) node[ right] {\footnotesize $\theta_{i,i-2}$} ;
\end{tikzpicture}
\caption{Construction of $C_v$}\label{fig:Cv}
\end{figure}

We also denote by $Conn(v)$ the connected component of $v$. The following lemma shows that the covering component of $v$ indeed covers the connected component of $v$ if $v$ is the closest vertex of the center in this component.
\begin{lemma}\label{lem:concomp}
A.a.s. for any $v\in B_O(R)$, if  $t_v=\max\set{t_u \mid u\in Conn(v)}$ the connected component of $v$ is included in $\EC_v$.
\end{lemma}
\begin{proof}
Suppose that the event $\mathcal{E}_R$ holds. This happens a.a.s. according to Lemma \ref{lem:sep}.

By contradiction, consider a vertex $u$ in the connected component of $v$ that is not contained in $\EC_v$, and a shortest path $v_0=v,\ldots,v_m=u$. Hence there exists a smallest $k \ge 1$ such that the vertex
$v_k$ is not in $\EC_v$. Let  $\calL_{i_k}$ be the layer of $v_k$.

 Suppose now there exists $k' < k$ so that for some $i_{k'} > i_{k}$, $v_{k'}\in \calL_{i_{k'}}$. We may then choose the largest $k'$, so that 
 $$\forall \ell\in\set{k'+1,\dots,k},\ \exists i_\ell\leq i_k,\quad v_\ell\in \calL_{i_\ell}\quad \text{(see Figure \ref{fig:concomp}).}$$

 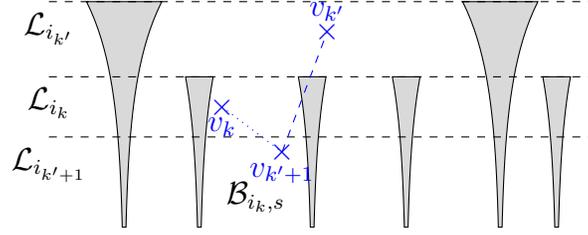
\begin{figure}[ht!]
\centering
\begin{tikzpicture}
 \newcommand{\sepzoneA}{
 \coordinate (A) at (-0.5,0)  ;
\coordinate (B) at (0.5,0)  ;
\coordinate (C) at ({exp(-3)/2},-3)  ;
\coordinate (D) at ({-exp(-3)/2},-3)  ;
  \fill[color=gray!30] (A)-- (B) -- plot [domain=0:-3] ({exp(\x)/2},\x) -- (D)  -- plot [domain=-3:0] ({-exp(\x)/2},\x) -- cycle ;
 \draw (A)-- (B) -- plot [domain=0:-3] ({exp(\x)/2},\x) -- (D)  -- plot [domain=-3:0] ({-exp(\x)/2},\x) -- cycle ;
 }
 
  \newcommand{\sepzoneB}{
   \coordinate (A) at ({-exp(-1)/2},-1)  ;
\coordinate (B) at ({exp(-1)/2},-1)  ;
\coordinate (C) at ({exp(-3)/2},-3)  ;
\coordinate (D) at ({-exp(-3)/2},-3)  ;
  \fill[color=gray!30] (A)-- (B) -- plot [domain=-1:-3] ({exp(\x)/2},\x) -- (D)  -- plot [domain=-3:-1] ({-exp(\x)/2},\x) -- cycle ;
 \draw (A)-- (B) -- plot [domain=-1:-3] ({exp(\x)/2},\x) -- (D)  -- plot [domain=-3:-1] ({-exp(\x)/2},\x) -- cycle ;
 }
 
\sepzoneA

\begin{scope}[shift={(5,0)}]
\sepzoneA
\end{scope}

%Aires de separation S_{i-1,k}
\begin{scope}[shift={(1,0)}]
\sepzoneB
\end{scope}

\begin{scope}[shift={(2.5,0)}]
\sepzoneB
\end{scope}

\begin{scope}[shift={(3.75,0)}]
\sepzoneB
\end{scope}

\begin{scope}[shift={(5.75,0)}]
\sepzoneB
\end{scope}

\draw[dashed] (-1,0) node[below]{$\calL_{i_{k'}}$}-- (6.25,0) ;
\draw[dashed] (-1,-1) node[below]{$\calL_{i_k}$} -- (6.25,-1) ;
\draw[dashed] (-1,-1.8) node[below]{$\calL_{i_{k'+1}}$} -- (6.25,-1.8) ;

\draw[dashed, blue]  (2.1,-2) node{\blue{$\times$}} node[below] {\blue{$v_{k'+1}$}} -- (2.7,-0.4) node{\blue{$\times$}} node[above] {\blue{$v_{k'}$}};
\draw[dotted, blue] (1.3,-1.4) node{\blue{$\times$}} node[below] {\blue{$v_{k}$}} --  (2.1,-2) ;

\draw (1.75,-2.6) node{$\calB_{i_k,s}$} ;

\end{tikzpicture}
\caption{Explanation of the proof of Lemma~\ref{lem:concomp}}\label{fig:concomp}
\end{figure}

\noindent Hence, the $v_{\ell}$ are in the same zone $\calB_{i_k,s}$ as $v_k$ and 
\begin{align*}
|\theta_{v_{k'}}-\theta_{v_k}| \leq \theta_{i_{k'}i_k}+5cR\theta_{i_{k}i_k}\leq 2\theta_{i_{k'}i_k}
\end{align*}
and thus
$v_k\in C_{v_{k'}}\subset \EC_v$ which is impossible. Thus necessarily $v=v_0$ is in the same layer as $v_k$ or in a layer closer to the boundary. Since $v$ is by hypothesis the vertex such that $t_v=\max\set{t_u \mid u\in Conn(v)}$, we must have $v  \in \calL_{i_{k}}$.  Therefore, by Observation~\ref{easy:calB}, $v$ and $v_k$ must be in the same zone $\calB_{i_k,s}$ for some $s$ and thus $v_k\in \EC_v$.
\end{proof}

\begin{lemma} \label{Concentration1} 
A.a.s.,
 $$
\forall 0\leq j<i\leq i_{\max},\  \forall v\in \calL_i,\ |\Theta_{i,j}(v)|\leq 4\max(8R,\e{|\Theta_{i,j}(t_i,0)|})
$$
 \end{lemma}

\begin{proof}
Let $d_{i,j}= 2\max(8R,\e{|\Theta_{i,j}(t_i,0)|})$. For each $0 \le j < i \le i_{\max}$, divide layer $\calL_j$ into $\lceil \pi/(2\theta_{i,j})\rceil$ sectors of angle (at most) $4\theta_{i,j}$. For any such sector $S_k^{(i,j)}$, for any $0 \le j < i \le i_{\max}$,
from Lemma~\ref{lem:Chernoff} we have
$$
 \p{|S_k^{(i,j)} \cap V \cap \calL_j|> d_{i,j}}\leq e^{-2R}.
$$
By a union bound over all $\lceil \pi/(2\theta_{i,j})\rceil$ sectors $S_k^{(i,j)}$ and then over all $i,j$, we have
$$
 \p{\exists 0 \leq j < i \leq i_{\max}, \exists 1\leq k\leq \lceil \pi/(2\theta_{i,j})\rceil,\ |S_k^{(i,j)} \cap V \cap \calL_j|> d_{i,j}}\leq R^2e^{R/2}e^{-2R}=o(1).
$$
Hence, since for each vertex $v \in V \cap \calL_i$, the set $\Theta_{i,j}(v)$ can intersect at most two adjacent sectors $S_k^{(i,j)}$, we have
\begin{align*}
  &\p{\exists 0\leq j<i\leq i_{\max},\  \exists v\in V \cap \calL_i,\ |\Theta_{i,j}(v)|> 2d_{i,j}}=o(1).
\end{align*}
\end{proof}

\begin{lemma} \label{lem:boundCv} There is a constant $K\geq1$ such that a.a.s.
 $$
 \forall 0\leq i\leq i_{\max},\ \forall v\in \calL_i,\ |\EC_v|\leq e^{2t_0+\frac{1}{2}t_i}.
$$
 \end{lemma}
\begin{proof}$ $

We first give an upper bound for $|C_v|$:
recall first that Lemma~\ref{Concentration1} says that  the event
 $$\mathcal{A}=\set{\forall 0\leq j<i\leq i_{\max},\  \forall v\in V \cap \calL_i,\ |\Theta_{i,j}(v)|\leq 4\max(8R,\e{|\Theta_{i,j}(t_i,0)|})}$$
 happens a.a.s. %with probability larger than $1-e^{-R}$ for large $R$.
  We now proceed  by induction on $i$ and prove that, on $\mathcal{A}$, for any $0\leq i\leq i_{\max}$,
  $$ 
 \forall j\leq i,\ \forall v\in V\cap \calL_j,\ |C_v|\leq  K e^{\frac{t_0+t_i}{2}}.
  $$
  for the constant $K=37\nu$.
  As for any $v\in V \cap \calL_0$, $C_v=\set{v}$, the result is obvious for $i=0$. Suppose now it is true for some $0\leq i< i_{\max}$. On the event $\mathcal{A}$, for any $v\in V \cap \calL_{i+1}$,
  \begin{align*}
   |C_v|&\leq 1+ \sum_{0\leq j\leq i}\sum_{u\in\Theta_{i+1,j}(v)}|C_u|\leq  1+|\Theta_{i+1,0}(v)|+ \sum_{1\leq j\leq i}\sum_{u\in\Theta_{i+1,j}(v)}|C_u|\\
   &\leq 1+4\max(8R,\e{|\Theta_{i+1,0}(t_{i+1},0)|})+\sum_{1\leq j\leq i}4\max(8R,\e{|\Theta_{i+1,j}(t_{i+1},0)|}) 19\nu e^{\frac{t_0+t_j}{2}}.
  \end{align*}
   According to Observation \ref{obs:Ltheta}, if $R$ is large enough, for $j=0$,
 \begin{align*}
 \e{|\Theta_{i+1,0}(t_{i+1},0)|}&=4\theta_{i+1,0}\e{|V\cap L_0|}\leq 9\nu e^{R/2-\frac12(R-t_{i+1}-t_0)}= 9\nu e^{\frac{t_0+t_{i+1}}{2}},
 \end{align*}
 and for $j\geq1$,  
 \begin{align*}
  \e{|\Theta_{i+1,j}(t_{i+1},0)|}&=4\theta_{i+1,j}\e{|V\cap L_j|}\leq 9 \nu e^{R/2-\alpha t_{j-1}-\frac12(R-t_{i+1}-t_j)}=9\nu e^{\frac{t_j+t_{i+1}}{2}-\alpha t_{j-1}}.
 \end{align*}
 Recall that $t_{-1}=0$ and 
 for $i\geq0$, $t_i=(\frac{4\alpha}{\alpha-1}+3i)\log R$.
 This leads to the following bound for $|C_v|$ for large $R$: 
  \begin{align*}
  |C_v|&\leq 1+36\nu e^{\frac{t_0+t_{i+1}}{2}}+ 32RK\sum_{1\leq j\leq i}e^{\frac{t_0+t_j}{2}}+ 36K\nu^2 e^{\frac{t_0+t_{i+1}}{2}} \sum_{1\leq j\leq i}e^{-\alpha t_{j-1}+t_j}\\
  &\\
  &\leq1+36\nu e^{\frac{t_0+t_{i+1}}{2}}\left(1+\frac{K}{\nu}R^{1/2}+\nu^2KR^{-\alpha}\right)\leq 37\nu  e^{\frac{t_0+t_{i+1}}{2}}
  %&\leq1+18\nu e^{\frac{t_0+t_{i+1}}{2}}\left(1+\frac{K}{\nu}Re^{-s/2}+\frac{K}{\nu}e^{-(\alpha-1) t_0+s}\right)
 \end{align*} 
 
Now we can proceed to obtain an upper bound for $|\EC_v|$: 
  For $i\in\set{0, \dots,i_{\max}}$, denote by $\Gamma_i$ the set 
 $$
 \Gamma_i=\set{v\in V\cap\calL_i \Bigm| |\theta_v|\leq 5cR\theta_{i,i}}.
 $$
 According to Observation \ref{obs:Ltheta}, there is a constant $K$ depending only on $\nu$ such that for $R$ large enough and $i\in\set{0, \dots,i_{\max}}$,
\begin{align*}
\e{|\Gamma_i|}&=10cR\theta_{i,i}\e{|V\cap \calL_i|}\leq K Re^{-\frac{1}{2}(R-2t_i)}e^{\frac{R}{2}-\alpha t_{i-1}}=K Re^{t_i-t_{i-1}-(\alpha-1)t_{i-1}}\leq K Re^{t_0}\leq \frac{1}{2}e^{\frac{3}{2}t_0}.
\end{align*}

Therefore,
\begin{align*}
 &\p{\mathcal{E}_R\cap\set{\exists i\in\set{0,\dots,i_{\max}},\ \exists k\in\set{1,\dots,k^i_{\max}},|V\cap \calL_i\cap \calB_{i,k}|\geq e^{\frac{3}{2}t_0}}}\\
 &\leq\sum_{i=0}^{i_{\max}}\left\lceil\frac{1}{2cR\theta_{i,i}}\right\rceil\p{|\Gamma_i|\geq e^{\frac{3}{2}t_0}}
 \end{align*}
 Now, since $|\Gamma_i|$ is a Poisson variable, Lemma~\ref{lem:Chernoff} says that 
 \begin{align*}
  \p{|\Gamma_i|\geq e^{\frac{3}{2}t_0}}\leq e^{-e^{\frac{3}{2}t_0}/8}.
 \end{align*}
Thus, the previous probability is smaller than
$$
\sum_{i=0}^{i_{\max}}\frac{1}{2cR}e^{\frac{R}{2}-t_i}e^{-e^{\frac{3}{2}t_0}/8}\leq e^{R/2-e^{\frac{3}{2}t_0}/8},
$$
which tends to 0 as $R$ tends to infinity.

Finally, a.a.s., for any $i\geq0$ and any $v\in V\cap\calL_i$, the cardinality of $\EC_v$ satisfies
\begin{align*}
 |\EC_v|&\leq \max_{u\in V\cap\calL_i} |C_u| \max_{k\leq k^i_{\max}} |V\cap \calL_i\cap \calB_{i,k}|\leq Ke^{\frac{t_0+t_i}{2}}e^{\frac32t_0},
\end{align*}
and the lemma follows.
\end{proof}

\begin{proof}[Proof of Theorem \ref{thm:main}]
%In the proof, the constants $C$ and $\widetilde{C}$ may change from line to line.
 According to Lemma~\ref{lem:boundCv}, there is a constant $K>0$ such that, a.a.s. 
 \begin{align}\label{majo1}
  \max_{v\in V}|Conn(v)|\leq\max_{v\in V}|\EC_v|&\leq e^{2t_0+\frac12t_{i_{\max}}}\leq e^{2t_0+\frac{t_{\max}+3\log R}{2}} =  e^{\frac{R}{4\alpha}+\left(\frac{8\alpha}{\alpha-1}+\frac32\right)\log R}.
 \end{align}
By Lemma~\ref{lem:concomp} we obtain the upper bound for $|L_1|$ in the theorem.
 
For the lower bound, by Lemma~\ref{lem:muBall}, for any function $\omega$ tending to infinity with $n$ arbitrarily slowly, $\mu(B_O(r_{\max} + \omega) \gg 1/n$, and hence a.a.s. we find a vertex $v$ with $t_v \ge t_{\max}-\omega$. 
In such case, the degree of $v$ is, by Lemma~\ref{lem:degree}, a.a.s. $\Theta(e^{\frac12(t_{\max}-\omega))})=n^{\frac{1}{2\alpha}+o(\omega/n)}$. The degree of a vertex is a lower bound on the size of its component, and hence Theorem~\ref{thm:main} follows. 

\end{proof}
\section*{Acknowledgements} The authors would like to thank Antoine Barrier for providing Figure~\ref{Antoine}.
\bibliographystyle{alpha}
\bibliography{biblio} 

\newcommand{\etalchar}[1]{$^{#1}$}
\begin{thebibliography}{KPK{\etalchar{+}}10}

\bibitem[ABF17]{ABF}
M.~A. Abdullah, M.~Bode, and N.~Fountoulakis.
\newblock Typical distances in a geometric model for complex networks.
\newblock {\em Internet Mathematics}, 1, 2017.

\bibitem[BFM15]{BFM15}
M.~Bode, N.~Fountoulakis, and T.~M\"{u}ller.
\newblock On the largest component of a hyperbolic model of complex networks.
\newblock {\em Electronic J.~of Combinatorics}, 22(3):P3.24, 2015.

\bibitem[BFM16]{BFM13b}
M.~Bode, N.~Fountoulakis, and T.~M\"{u}ller.
\newblock The probability of connectivity in a hyperbolic model of complex
  networks.
\newblock {\em Random Structures \& Algorithms}, 49(1):65--94, 2016.

\bibitem[BKL{\etalchar{+}}17]{BKLMM17}
K.~Bringmann, R.~Keusch, J.~Lengler, Y.~Maus, and A.R. Molla.
\newblock Greedy routing and the algorithmic small-world phenomenon.
\newblock In {\em Proceedings of the ACM Symposium on Principles of Distributed
  Computing}, PODC'17, pages 371--380, New York, NY, USA, 2017. ACM.

\bibitem[BLM13]{MR3185193}
S.~Boucheron, G.~Lugosi, and P.~Massart.
\newblock {\em Concentration inequalities}.
\newblock Oxford University Press, Oxford, 2013.
\newblock A nonasymptotic theory of independence, With a foreword by Michel
  Ledoux.

\bibitem[BnPK10]{BPK10}
M.~Bogu\~{n}\'{a}, F.~Papadopoulos, and D.~Krioukov.
\newblock Sustaining the internet with hyperbolic mapping.
\newblock {\em Nature Communications}, 1:62, 2010.

\bibitem[CF16]{CF16}
E.~Candellero and N.~Fountoulakis.
\newblock Clustering and the hyperbolic geometry of complex networks.
\newblock {\em Internet Mathematics}, 12(1--2):2--53, 2016.

\bibitem[FK15]{fk15}
T.~Friedrich and A.~Krohmer.
\newblock On the diameter of hyperbolic random graphs.
\newblock In {\em Automata, Languages, and Programming - 42nd International
  Colloquium -- {ICALP} Part {II}}, volume 9135 of {\em {LNCS}}, pages
  614--625. Springer, 2015.

\bibitem[FM18]{FMLaw}
N.~Fountoulakis and T.~M\"{u}ller.
\newblock Law of large numbers for the largest component in a hyperbolic model
  of complex networks.
\newblock {\em Annals of Applied Probability}, 28:607--650, 2018.

\bibitem[GPP12]{GPP12}
L.~Gugelmann, K.~Panagiotou, and U.~Peter.
\newblock Random hyperbolic graphs: Degree sequence and clustering.
\newblock In {\em Automata, Languages, and Programming - 39th International
  Colloquium -- ICALP Part II}, volume 7392 of {\em {LNCS}}, pages 573--585.
  Springer, 2012.

\bibitem[KL]{KL19+}
J.~Komjathy and B.~Lodewijks.
\newblock Explosion in weighted hyperbolic random graphs and geometric
  inhomogeneous random graphs.
\newblock {\em Stochastic Processes and its Applications, to appear}.

\bibitem[KM15]{km15}
M.~Kiwi and D.~Mitsche.
\newblock A bound for the diameter of random hyperbolic graphs.
\newblock In {\em Proceedings of the 12th Workshop on Analytic Algorithmics and
  Combinatorics -- {ANALCO}}, pages 26--39. {SIAM}, 2015.

\bibitem[KM18]{KM18}
M.~Kiwi and D.~Mitsche.
\newblock Spectral gap of random hyperbolic graphs and related parameters.
\newblock {\em Annals of Applied Probability}, 28:941--989, 2018.

\bibitem[KM19]{KM19+}
M.~Kiwi and D.~Mitsche.
\newblock On the second largest component of random hyperbolic graphs.
\newblock {\em SIAM Journal on Discrete Mathematics}, 33(4):2200--2217, 2019.

\bibitem[KPK{\etalchar{+}}10]{KPKVB10}
D.~Krioukov, F.~Papadopoulos, M.~Kitsak, A.~Vahdat, and M.~Bogu\~{n}\'{a}.
\newblock Hyperbolic geometry of complex networks.
\newblock {\em Phyical Review E}, 82(3):036106, 2010.

\bibitem[MS19]{MStaps}
T.~M\"{u}ller and M.~Staps.
\newblock The diameter of {KPKVB} random graphs.
\newblock {\em Advances of Applied Probability}, 51(2):358--377, 2019.

\bibitem[Pen03]{MP03}
M.~Penrose.
\newblock {\em Random geometric graphs}.
\newblock Oxford University Press, 2003.

\bibitem[TM67]{TM67}
J.~Travers and S.~Milgram.
\newblock The small world problem.
\newblock {\em Psychology Today}, 1(1):61--67, 1967.

\end{thebibliography}

\end{document}